\documentclass[11pt]{amsart}

\usepackage[colorlinks=true, pdfstartview=FitV, linkcolor=blue, 
citecolor=blue]{hyperref}

\usepackage{amssymb,amsmath,amscd}
\usepackage{accents}
\usepackage{stackrel}
\usepackage{graphicx}
\usepackage{bbm}
\usepackage{a4wide}
\usepackage{epsfig,psfrag}  % added psfrag
\usepackage{pstricks}
\usepackage{pstricks-add}
\usepackage{faktor}
\usepackage{amsthm}
\usepackage{comment}
\usepackage{stmaryrd}

\makeatletter
\newcommand\xleftrightarrow[2][]{%
  \ext@arrow 9999{\longleftrightarrowfill@}{#1}{#2}}
\newcommand\longleftrightarrowfill@{%
  \arrowfill@\leftarrow\relbar\rightarrow}
\makeatother

\theoremstyle{plain}
\newtheorem{theorem}{Theorem}

\newtheorem{lemma}[theorem]{Lemma}
\newtheorem{coro}[theorem]{Corollary}

\theoremstyle{definition}
\newtheorem{remark}[theorem]{Remark}
\newtheorem{example}[theorem]{Example}
\newtheorem{deff}[theorem]{Definition}

\newcommand{\ts}{\hspace{0.5pt}}
\newcommand{\nts}{\hspace{-0.5pt}}

\newcommand{\RR}{\mathbb{R}\ts}

\newcommand{\EE}{\mathbb{E}}

\newcommand{\cA}{\mathcal{A}}
\newcommand{\cB}{\mathcal{B}}
\newcommand{\cC}{\mathcal{C}}

\newcommand{\cG}{\mathcal{G}}

\newcommand{\cP}{\mathcal{P}}
\newcommand{\cQ}{\mathcal{Q}}
\newcommand{\cR}{\mathcal{R}}
\newcommand{\cS}{\mathcal{S}}

\newcommand{\cX}{\mathcal{X}}

\newcommand{\bP}{\boldsymbol{P}}

\newcommand{\pmin}{\ts\ts\underline{\nts\nts 0\nts\nts}\ts\ts}
\newcommand{\pmax}{\ts\ts\underline{\nts\nts 1\nts\nts}\ts\ts}

\newcommand{\RNum}[1]{\uppercase\expandafter{\romannumeral #1\relax}}
\newcommand{\trans}{{\raisebox{1pt}{$\scriptscriptstyle\mathsf{T}$}}}

\newcommand{\defeq}{\mathrel{\mathop:}=}

\newcommand{\ee}{\mathrm{e}}

\newcommand{\dd}{\, \mathrm{d}}

\newcommand{\udo}[1]{\underaccent{$\text{.}$}{#1\ts}\nts}

\newcommand{\todo}[1]{{\tt TODO:#1}}

\begin{document}

\definecolor{lil}{rgb}{.44,.26,.44} %hexa: 8f008f
\definecolor{lila}{rgb}{.01,.01,.46} %hexa: 020275
%\definecolor{blo}{rgb}{0.11,0.67,0.75} %hexa 1cabbf
\definecolor{blo}{cmyk}{.1,.1,.01,.1}
\definecolor{gre}{rgb}{.55,.78,0} % hexa 8cc700
%\definecolor{ora}{rgb}{1,.58,0} % hexa ff9400
%\definecolor{blu}{rgb}{0,.64,0.78} % hexa 00a3c7
\definecolor{gre}{rgb}{.06,.49,0.03} % hexa 107d08
\definecolor{ora}{rgb}{.84,.40,0} % hexa d66700 

\title[Genetic recombination as a Generalised Gradient Flow] {Genetic recombination as a Generalised Gradient Flow}

\author{Frederic Alberti}

\address{Fakult\"at f\"ur Mathematik, Universit\"at Bielefeld, \newline
\hspace*{\parindent}Postfach 100131, 33501 Bielefeld, Germany}
\email{falberti@math.uni-bielefeld.de}

\begin{abstract}
It is well known that the classical recombination equation for two parent individuals is equivalent to the law of mass action of a strongly reversible chemical reaction network, and can thus be reformulated as a generalised gradient system. 
Here, this is generalised to the case of an arbitrary number of parents. Furthermore, the gradient structure of the backward-time partitioning process is investigated.
\end{abstract}

\maketitle

\section{Introduction}
\label{sec:introduction}

Genetic recombination describes the reshuffling of genetic information that occurs during the reproductive cycle of (sexual) organisms; it is one of the major mechanisms that maintain genetic diversity within populations. One of the standard models for its description is the \emph{deterministic recombination equation in continuous time}, in the following simply referred to as \emph{recombination equation}; for background, see~\cite{buerger}.
This equation describes the evolution of the distribution of types in a (haploid) population under the assumption of random mating, while neglecting stochastic fluctuations.  

In the case of finite sets of alleles at an arbitrary (but finite) number of sites, the recombination equation was reinterpreted by~\cite{muellerhofbauer} as the law of mass action of a network of chemical reactions between gametes. This reaction network was shown to be strongly reversible and general theory~\cite{mielke,yong} on chemical reaction networks thus implies that it admits a representation in terms of a generalised gradient system, with respect to entropy~\cite{hofbauer}. This strengthens an earlier result by Akin~\cite[Thm.~\RNum{3}.2.5]{akin} that entropy is a strong Lyapunov function for recombination; a somewhat weaker statement can be found in~\cite[Thm. 6.3.5]{lyubich}.

A generalised version of the model, which allows for a more general reproduction mechanism (involving an arbitrary number of parents) as well as more general (not necessarily discrete) type spaces, was considered in~\cite{baakebaakesalamat}. There, the authors reduced the original measure-valued, infinite-dimensional equation via a suitable ansatz function to a finite-dimensional, albeit still nonlinear, system. This system was then analysed using lattice-theoretic techniques, leading to an explicit recursion formula for its solution. A somewhat different, but simpler approach can be found in~\cite{baakebaake}, which relates the evolution of the type distribution, forward in time, to an ancestral partitioning process backward in time. More precisely, the solution of the recombination equation is expressed in terms of the solution of the (linear) differential equation for the law of the partitioning process.

One obvious question is now whether this more general model can also be represented as a strongly reversible chemical reaction network, and, consequently, as a generalised gradient system like the more classical version, which involves only two parents and finite type spaces. 

We shall see that, in the case of finite type spaces, the answer is affirmative, generalising the results of M\"uller and Hofbauer~\cite{muellerhofbauer} to the multi-parent case. In addition, we will that the dynamics of the law of the partitioning process~\cite{baakebaake,baakebaakereview}, which is independent of the type space, can be rewritten as as a generalised (linear) gradient system. 

Finally, we reconsider the finite-dimensional, nonlinear system from~\cite{baakebaakesalamat} and show that it, too, can be rewritten in terms of a law of mass action of a chemical reaction network, which structurally resembles the one mentioned above for finite type spaces; however, it is not clear whether it is indeed a gradient system, due to the loss of reversibility incurred by the loss of information when transitioning from types to partitions.

The paper is organised as follows. First, we recall and explain the recombination equation for multiple parents. The connection to chemical reaction networks is established in Section~\ref{sec:networks}, for finite type spaces. Section~\ref{sec:gradients} contains the results on the gradient structure of this network, and Section~\ref{sec:markov} explains the gradient structure of a particular class of Markov chains, which covers the partitioning process. Finally, Section~\ref{sec:nonlinearpartitioning} contains the reformulation of the nonlinear system from~\cite{baakebaakesalamat} as a law of mass action.

\section{The recombination equation}
\label{sec:recoeq}
For our purposes, a \emph{genetic type} is a sequence
$
x = (x_1,\ldots,x_n) \in X \defeq  X_1 \times \cdots \times X_n
$
of fixed length $n$, where $X_1,\ldots,X_n$ are locally compact Hausdorff spaces and the evolution of the \emph{$($gametic$)$ type distribution} of the population is modelled as a differentiable one-parameter family $\omega = (\omega_t^{})_{t \geqslant 0}$ of (Borel) probability measures on $X$. 
This generality is useful in the context of quantitative genetics, when modelling the evolution of quantitative, polygenic traits such as body size, brain volume, growth rate or milk production; cf.~\cite[Ch. IV]{buerger}.

To understand the dynamics of $\omega$, let us start on the level of individuals. When two or more parents jointly produce an offspring, a partition of $S := \{1,\ldots,n\}$, the set of genetic \emph{sites} (or \emph{loci}), describes how the type of that offspring is pieced together from the types of its parents; recall that a partition of a set $M$ is a collection of pairwise disjoint, non-empty subsets, called \emph{blocks}, whose union is $M$; whenever we enumerate the blocks of a partition of $S$, that is, whenever we write
$
\cA = \{A_1,\ldots,A_{|\cA|}^{}\},
$
where $|\cA|$ is the number of blocks in $\cA$, we order the blocks such that $\cA_1$ is the block that contains $1$ and, for all $2 \leqslant k \leqslant |\cA|$, $A_k$ is the block that contains the smallest element not contained in $\bigcup_{j = 1}^{k-1} A_j$. 

Formally, given $k$ parent individuals of types $x^{(1)},\ldots,x^{(k)}$, (where $x^{(i)} = \big (x^{(i)}_1,\ldots,x^{(i)}_n \big)$ for all $1 \leqslant i \leqslant k$), and a partition $\cA = \{A_1,\ldots,A_k \}$ of $S$ into $k$ blocks, the type of the offspring is
\begin{equation} \label{exampleoffspring}
y = \pi_{A_1}^{} \big (x^{(1)} \big) \sqcup \ldots \sqcup \pi_{A_k}^{} \big (x^{(k)} \big).
\end{equation}
Here, for each $A \subseteq S$,  $\pi_{A}^{}$ is the projection
$
(x_i)_{i \in S} \mapsto (x_i)_{i \in A}
$
and the symbol $\sqcup$ is used to denote the joining of gene fragments, respecting the order of the sites; given pairwise disjoint subsets $U_1,\ldots,U_k$ of $S$ and sequences $x^{(\ell)}$ in $\pi^{}_{U_\ell} (X)$ with $1 \leqslant \ell \leqslant k$, we write
$
x^{(1)} \sqcup \ldots \sqcup x^{(k)} 
$
for the sequence indexed by $U_1 \cup \ldots \cup U_k$ that has at site $i$ the letter $x^{(j)}_i$, where $U_j$ is the unique subset that contains $i$. In particular, in Eq.~\eqref{exampleoffspring}, 
$y = (y_1,\ldots,y_n)$ where $y_i = x^{(j)}_i$ if $i \in U_j$. 

For any $U \subseteq S$, we denote by 
\begin{equation*}
X_U^{} \defeq \prod_{i \in U} X_i
\end{equation*}
the marginal type space with respect to $U$.

%Here, we adopt the usual convention of ordering the blocks in a partition such that the first block is always the block containing $1$ and the block $\ell + 1$ is the one that contains the smallest number that is not contained in the first  $\ell$ blocks. 

In the following, we denote by $\bP(S)$ the set of all partitions of $S$, not to be confused with the set $\cP(X)$ of all (Borel) probability measures on $X$.    

To express the effect of recombination on the type distribution in a concise way, we define for any $\cA \in \bP(S)$ a (nonlinear) operator on $\cP(X)$ by
\begin{equation*}
\cR_\cA (\nu) \defeq \bigotimes_{i=1}^{|\cA|}\pi^{}_{A_i} . \nu.
\end{equation*}
Here, $\bigotimes$ denotes measure product and the dot denotes the push-forward of probability measures; i.e.,
$
\pi_A^{} . \nu (E) \defeq \nu \big (\pi^{-1}_A (E) \big ),
$
for any Borel measurable subset $E \subseteq X_A^{}$ and $\nu \in \cP(X)$. The operators $\cR_\cA^{}$ are called recombinators~\cite{baakebaakecan}. Clearly, $\cR_\cA^{} (\nu)$ is the distribution of the type of the joint offspring of $|\cA|$ parents whose types are drawn independently from $\nu$, where the letters at two different sites $k$ and $\ell$ come from the same  parent if and only if $k$ and $\ell$ are in the same block of $\cA$; we call such an offspring \emph{$\cA$-recombined}.
If $X$ is finite, the recombinator can also be written as follows.
\begin{lemma}\label{vectorrecombinator}
Assume that $X$ is finite. Then,  for all $\nu \in \cP(X)$ and all \mbox{$\cA = \{A_1,\ldots,A_k\} \in \bP(S)$}, we have
\begin{equation*}
\cR_{\cA} (\nu) =  \sum_{x^{(1)}_{},\ldots,x^{(|\cA|)}_{} \in X} \nu \big (x^{(1)}_{} \big )\cdot \ldots \cdot \nu \big (x^{(|\cA|)}_{} \big) \delta_{\bigsqcup_{i = 1}^{|\cA|} \pi_{A_i}^{} (x^{(i)} )}^{}.
\end{equation*}
\end{lemma}
\begin{proof}
Let us write $\tilde{\cR}$ for the map on $\cP(X)$ defined by the right-hand side. Then, for all $y \in X$,
\begin{equation*}
\tilde{\cR}(\nu) (y) = \sum_{\substack{x^{(1)}_{},\ldots,x^{(k)}_{} \in X \\ \pi^{}_{A_i} (x^{(i)}) = \pi^{}_{A_i}(y) \, \forall i}} \nu \big (x^{(1)}_{} \big )\cdot \ldots \cdot \nu \big (x^{(|\cA|)}_{} \big) = \prod_{i = 1}^k  \nu\big(\pi_{A_i}^{-1}(x)\big),
\end{equation*}
which implies the identity claimed.
\end{proof}

\begin{remark}\label{pointsandmeasures}
As expressions of the form
\begin{equation*}
\delta^{}_{\bigsqcup_{i=1}^k \pi_{A_i}^{} ( x^{(i)})}
\end{equation*}
are quite cumbersome, we simplify the notation by formally identifying each element $m$ in a finite set $M$ with the associated point (or Dirac) measure $\delta_m$. Under this convention, the statement from Lemma~2.2 reads 
\begin{equation*}
\cR_{\cA} (\nu) =  \sum_{x^{(1)}_{},\ldots,x^{(|\cA|)}_{} \in X} \nu \big (x^{(1)}_{} \big )\cdot \ldots \cdot \nu \big (x^{(|\cA|)}_{} \big) \bigsqcup_{i = 1}^{|\cA|} \pi_{A_i}^{} (x^{(i)} ). 
\end{equation*}
Put differently, we identify the vector space of finite signed measures on $M$ with the vector space $\RR^M$ of formal sums of its elements. Unless stated otherwise, all vectors are interpreted as column vectors. This entails  that the standard scalar product $\langle v, w \rangle$ of any two vectors $v$ and $w$ can be written as $v^\trans w$ (where $\trans$ denotes transposition) whereas $v w^\trans$ denotes the matrix that maps any other vector $u$ to $\langle w,u \rangle v$.
\hfill $\diamondsuit$
\end{remark}

Before we continue, we will need to recall from~\cite{baakebaake} a few additional notions around partitions. Given two partitions $\cA$ and $\cB$, we say that $\cA$ \emph{is finer than} $\cB$ and write $\cA \preccurlyeq \cB$ if and only if every block of $\cA$ is a subset of some block of $\cB$; this defines a partial order on $\bP(S)$, and we denote the unique minimal (maximal) element by $\pmin \defeq \{\{i\} \mid i \in S\}$ ($\pmax \defeq \{S\}$). The \emph{coarsest common refinement} of $\cA$ and $\cB$ is defined as
\begin{equation*}
\cA \wedge \cB := \{A \cap B \mid A \in \cA, B \in \cB\} \setminus \varnothing;
\end{equation*}
it is the largest (coarsest) element of $\bP(S)$ smaller (finer) or equal to both $\cA$ and $\cB$. Furthermore, given a partition $\cA$ of $S$ and some subset $U \subseteq S$, we denote by 
\begin{equation*}
\cA|_U := \{ A \cap U \mid A \in \cA \} \setminus \varnothing 
\end{equation*} 
the partition of $U$ \emph{induced} by $\cA$.

We are now ready to state the \emph{recombination equation} \cite[Eq. (7)]{baakebaakesalamat},
\begin{equation}\label{recoeq}
\dot{\omega_t} = \sum_{\cA \in \bP(S)} \varrho(\cA) \big (\cR_\cA (\omega_t) - \omega_t \big ),
\end{equation}
where the $\varrho(\cA)$ are non-negative real numbers, called \emph{recombination rates}.
This equation expresses that in each infinitesimal time interval $[t,t + \dd t]$, for each $\cA \in \bP(S)$, each individual is with probability $\varrho(\cA) \dd t$ replaced by a new $\cA$-recombined offspring, distributed as $\cR_\cA^{}(\omega_t^{})$. In other words, the current type distribution $\omega_t^{}$ is replaced by the convex combination 
\begin{equation}\label{convexcombination}
\Big (1 - \sum_{\cA \in \bP(S)} \varrho(\cA) \dd t \Big ) \omega_t + \sum_{\cA \in \bP(S)} \varrho(\cA) \big (\cR^{}_{\cA} (\omega^{}_t) \big ) \dd t.
\end{equation}

\begin{remark}
For the reader familiar with stochastic models for finite population size, we remark that Eq.~\eqref{recoeq} may alternatively be obtained from the Moran model with recombination via a dynamic law of large numbers. This is because the Moran models with growing population size form a so-called \emph{density dependent family}; see~\cite[Thm. 11.2.1]{eundk}. \hfill $\diamondsuit$ 
\end{remark} 

Eq.~\eqref{convexcombination} motivates the ansatz
\begin{equation}\label{ansatz}            
\omega_t = \sum_{\cA \in \bP(S)} a_t(\cA) \cR^{}_\cA (\omega_0)
\end{equation}
for the solution of Eq.~\eqref{recoeq}.
Now, inserting this ansatz into Eq.~\eqref{recoeq}, leads to the following result.
\begin{theorem}[{\cite[Thm.~1]{baakebaakesalamat}}]\label{reduction}
Every solution $\omega$ of Eq.~\eqref{recoeq} has the form
\begin{equation}\label{ansatz}
\omega_t^{} = \sum_{\cA \in \bP(S)} a_t(\cA) \cR_{\cA}^{}(\omega^{}_0),
\end{equation}
where the coefficients $a_t(\cA)$ satisfy the coupled nonlinear differential equations
\begin{equation*}
\dot{a}_t(\cA)  = -\sum_{\cB} \varrho(\cB) \cdot a_t(\cA)+ \sum_{\udo{\cB} \succcurlyeq \cA} \Bigg ( \prod_{i=1}^{|\cB|}\sum_{\substack{\cC \in \bP(S) \\ \cC|_{B_i} = \cA|_{B_i}}} a_t(\cC) \Bigg) \varrho (\cB),
\end{equation*}
with initial value $a_0(\pmax) = 1$ and $a_0(\cA) = 0$, otherwise. The sums run over all partitions of $U$, where the underdot marks the summation variable.\qed
\end{theorem}
We may also (compare Remark~\ref{pointsandmeasures}) rewrite this system in vector-notation as follows.
\begin{equation}\label{nonlinearodes}
\dot{a}_t  = -\sum_{\cA} \sum_{\cB} \varrho(\cB) \cdot a_t(\cA) \cA + \sum_{\cA} \sum_{\udo{\cB} \succcurlyeq \cA} \Bigg ( \prod_{i=1}^{|\cB|}\sum_{\substack{\cC \in \bP(S) \\ \cC|_{B_i} = \cA|_{B_i}}} a_t(\cC) \Bigg) \varrho (\cB) \cA,
\end{equation} 
where $a_t \defeq \sum_{\cA \in \bP(S)} a_t(\cA) \cA = \sum_{\cA \in \bP(S)} a_t(\cA) \delta_\cA^{}$.
This system can be solved recursively by lattice-theoretic means; compare~\cite{baakebaakesalamat}. We will consider it in greater detail  in Section~\ref{sec:nonlinearpartitioning} and show that it is the law of mass action of a chemical reaction network (compare Section~\ref{sec:networks}).

While the system~\eqref{nonlinearodes} is finite-dimensional, it is still highly nonlinear. In fact, Eq.~\eqref{recoeq} can also be related to a \emph{linear} system, via an ancestral partitioning process that runs \emph{backward} in time. Here, we only give a brief sketch of the idea; for further background on genealogical methods in the context of recombination, the reader is referred to the excellent review~\cite{baakebaakereview}.

The partitioning process is a Markov chain $\Sigma = (\Sigma_t)_{t \geqslant 0}$ in continuous time with state space $\bP(S)$ and is most easily understood when started in $\Sigma_0 = \pmax$. Assume that we want to sample the type of a single individual (Alice, say) that is alive at time $T$; Now, for $0 \leqslant t \leqslant T$, each block $\sigma$ of $\Sigma_t$ corresponds to an independent ancestor of Alice that lived at time $T - t$ and from whom she inherited the letters at the sites in $\sigma$. At time $T = T - 0$, Alice herself is alive and corresponds to the unique block of $\Sigma_0^{} =\pmax$.

To understand the time-evolution of $\Sigma$, keep in mind that every block in $\Sigma_{t}$ corresponds to one of Alice's ancestors. Recall that, by our interpretation of Eq.~\eqref{recoeq}, every individual alive at time $T - t$ was with probability $\varrho(\cB) \dd t$ a $\cB$-recombined offspring of parents alive at time $T - t - \dd t$. This means for the evolution of the partitioning process that a block $A \in \Sigma_t$ is in the infinitesimal time step from $t$ to $t + \dd t$, with probability $\varrho(\cB) \dd t$, replaced by the collection of blocks of induced partition $\cB_A^{}$, which reflects the partitioning of the genome of the corresponding ancestor of Alice across its own parents.
More formally, $\Sigma$ is a Markov chain in continuous time with rate matrix  $\cQ$, where
\begin{equation}\label{markovgenerator}
Q(\cA, \cB) \defeq \left \{ \begin{array}{ll}
0, & \text{if } \cB \not \preccurlyeq \cA, \\
\varrho^{A}_{\cB^{}_A}, &\text{if } \cB = \big ( \cA \setminus \{A\} \big ) \cup \cB_A^{} \\
-\sum_{\udo{\cC} \neq \cA} Q(\cA,\cC),  &\text{otherwise.}
\end{array} \right . 
\end{equation}
Here, the \emph{marginal recombination rates} are given by
\begin{equation*}
\varrho^{A}_{\cB^{}_A} \defeq \sum_{\substack{ \cC \in \bP(S) \\ \cC|^{}_{A} = \cB_A^{}}} \varrho_\cC^{}.
\end{equation*}
Formalising our verbal discussion above gives the following stochastic representation
\begin{equation}\label{stochasticrepresentation}
\omega_t^{} = \EE \big [ \cR_{\Sigma_t^{}}^{} (\omega_0) \mid \Sigma_0^{} = \pmax \big]
\end{equation}
of the solution $\omega$ of Eq.~\ref{recoeq}.
More generally for arbitrary starting values, one has the \emph{duality relation}
\begin{equation*}
\cR_\cA^{} (\omega_t) =  \EE \big [ \cR_{\Sigma_t^{}}^{} (\omega_0) \mid \Sigma_0^{} = \cA \big].
\end{equation*}
Put differently, Eq.~\eqref{stochasticrepresentation} implies that any solution of Eq.~\eqref{recoeq} is of the form 
\begin{equation*}
\omega_t^{} = \sum_{\cA \in \bP(S)} b_t(\cA) \cR^{}_{\cA} (\omega_0),
\end{equation*}
where the vector $b_t$ solves the initial-value-problem of the \emph{linear} ode
\begin{equation*}
\dot{b}_t = \cQ^\trans b_t.
\end{equation*}
with initial value $\pmax$.

\section{Chemical reaction networks}
\label{sec:networks}
Let us recapitulate a few basic notions in chemical reaction network theory, taylored to our purposes. For an introduction, see~\cite{feinberg}.
 
Let $\cS$ (not to be confused with $S$, the set of sequence sites) be a finite set, the elements of which will be thought of as the \emph{reacting species} in a \emph{chemical reaction network} (CRN), that is, a finite collection of \emph{chemical reactions}, which are represented by symbolic expressions of the form 
\begin{equation}\label{formalreaction}
r_1 + \ldots + r_{m_1} \xrightarrow{\, \kappa \,}  s_1 + \ldots + s_{m_2}.
\end{equation}
Here, the $r_i$ and $s_i$ are reacting species (not necessarily distinct) and $\kappa > 0$ is the \emph{reaction constant}. The left and right-hand sides in Eq.~\eqref{formalreaction} are called the \emph{complexes} of \emph{substrates} and \emph{products}. In our setting, we will always have $m_1 = m_2 = m$, as we will see later.

\begin{remark}
Recall from Remark~\ref{pointsandmeasures} that we formally identified the elements of any finite set with the corresponding point (or Dirac) measures. In this sense, the addition in Eq.~\eqref{formalreaction} can be understood as addition of vectors in the space of signed measures on $\cS$. 
\end{remark}

Of particular interest are \emph{strongly reversible} CRNs. They are usually defined as CRNs in which the forward reaction constant agrees with the backward reaction constant for every reaction. In the present setting, where we think of reactions as unidirectional, it is more convenient to phrase this slightly differently.
\begin{deff}\label{reversibility}
A CRN is called \emph{strongly reversible} if it can be partitioned into pairs, each consisting of a reaction,
\begin{equation*}
r_1 + \ldots + r_m \xrightarrow{\, \kappa \,}  s_1 + \ldots + s_m,
\end{equation*}
together with its backward reaction,
\begin{equation*}
s_1 + \ldots + s_m \xrightarrow{\, \kappa \,} r_1 + \ldots + r_m.
\end{equation*}
Note that the reaction constant is the same for both reactions.\hfill $\diamondsuit$
\end{deff}

Given a CRN, it is natural to inquire about the dynamics of the probability vector
\begin{equation*}
c_t = \sum_{s \in \cS} c_t(s) \delta_s = \sum_{s \in \cS} c_t(s) s
\end{equation*}
of normalised concentrations of species. As the left and right-hand sides in Eq.~\eqref{formalreaction} contain the same number of reacting species, the total mass is preserved and may therefore be normalised to one. 

The \emph{law of mass action} translates the collection of formal expressions~\eqref{formalreaction} into a system of coupled differential equations for $c = (c_t)_{t \geqslant 0}$. It assumes that each reaction occurs with a rate that is proportional to the concentration of each of the substrates, and hence to their product; the proportionality factor is the reaction constant $\kappa$ in Eq.~\eqref{formalreaction}. As each reaction decreases the concentration of substrates and increases the concentration of products, we obtain the following system of ordinary differential equations,
\begin{equation*}
\dot{c_t} = \sum \kappa c_t(r_1) \cdot \ldots \cdot c_t(r_m) \big ( s_1 + \ldots + s_m - r_1 - \ldots - r_m \big ),
\end{equation*}
where, again, the reacting species $s_1,\ldots,s_m$ and $r_1,\ldots,r_m$ are identified with the corresponding point measures $\delta^{}_{s_1},\ldots,\delta^{}_{s_m}$ and $\delta^{}_{r_1},\ldots,\delta^{}_{r_m}$ and the sum is taken over all reactions that make up the CRN. 
We refer the interested reader to~\cite[Ex.~11.1.C]{eundk} for a probabilistic variation on this theme. 

We now return to recombination. In~\cite{muellerhofbauer}, genetic recombination is treated as a CRN with the types as reacting species, in the special case of two parents. For example, recombination according to $\cA = \{\{1,2\},\{3\}\}$ translates to the reaction
\begin{equation*}
(\textcolor{gre}{x_1,x_2,x_3}) + (\textcolor{lil}{y_1,y_2,y_3}) \xrightarrow{\, \frac{\varrho(\cA)}{2} \,} (\textcolor{gre}{x_1,x_2},\textcolor{lil}{y_3}) + (\textcolor{lil}{y_1,y_2},\textcolor{gre}{x_3}).
\end{equation*}
This describes the process of recombination at the molecular level; first, the parental sequences $(\textcolor{gre}{x_1,x_2,x_3})$ and $(\textcolor{lil}{y_1,y_2,y_3})$ are split in two, according to $\cA$. Then, two new sequences are obtained by joining the leading part of one sequence with the trailing part of the other, and vice versa. For each (ordered) pair of types and each partition $\cA$, the reaction constant is $\frac{\varrho(\cA)}{2}$; this is a special case of Theorem~\ref{networkrepresentationgam}, which is stated below.
In the case when there are more than two parents, the basic idea remains the same; for any partition $\cC$, take $|\cC|$ types, split each of them according to $\cC$, rearrange the parts and join them back together. Note that this last step is somewhat ambiguous; already in the three-parent case, this can be done in at least two different ways; either,
\begin{equation}\label{onepossibility}
(\textcolor{gre}{x_1,x_2,x_3}) + (\textcolor{lil}{y_1,y_2,y_3}) + (\textcolor{ora}{z_1,z_2,z_3}) \longrightarrow (\textcolor{gre}{x_1},\textcolor{lil}{y_2},\textcolor{ora}{z_3}) + (\textcolor{lil}{y_1},\textcolor{ora}{z_2},\textcolor{gre}{x_3}) + (\textcolor{ora}{z_1},\textcolor{gre}{x_2},\textcolor{lil}{y_3}),
\end{equation}
or
\begin{equation}\label{anotherpossibility}
(\textcolor{gre}{x_1,x_2,x_3}) + (\textcolor{lil}{y_1,y_2,y_3}) + (\textcolor{ora}{z_1,z_2,z_3}) \longrightarrow (\textcolor{gre}{x_1},\textcolor{ora}{z_2},\textcolor{lil}{y_3}) + (\textcolor{ora}{z_1},\textcolor{lil}{y_2},\textcolor{gre}{x_3}) + (\textcolor{lil}{y_1},\textcolor{gre}{x_2},\textcolor{ora}{z_3}).
\end{equation}
One way to resolve this ambiguity is to order the substrates and define the reaction accordingly. Thus, there may be many different reactions with a common complex of substrates. More precisely, for every $\cC \in \bP(S)$ and each ordered tuple \mbox{$\big(x^{(1)},\ldots,x^{(|\cC|)} \big) \in X^{\cC}$}, we define a chemical reaction via the following graphical construction, illustrated in Figure~\ref{reactionscheme}. 
\begin{figure}[t]
\psfrag{1}{$\pi^{}_{C_1}\big (x^{(1)} \big)$}
\psfrag{2}{$\pi^{}_{C_1}\big (x^{(2)} \big)$}
\psfrag{3}{$\pi^{}_{C_1}\big (x^{(3)} \big)$}
\psfrag{4}{$\pi^{}_{C_1}\big (x^{(1)} \big)$}
\psfrag{5}{$\pi^{}_{C_1}\big (x^{(2)} \big)$}

\psfrag{6}{$\pi^{}_{C_2}\big (x^{(1)} \big)$}
\psfrag{7}{$\pi^{}_{C_2}\big (x^{(2)} \big)$}
\psfrag{8}{$\pi^{}_{C_2}\big (x^{(3)} \big)$}
\psfrag{9}{$\pi^{}_{C_2}\big (x^{(1)} \big)$}
\psfrag{10}{$\pi^{}_{C_2}\big (x^{(2)} \big)$}

\psfrag{11}{$\pi^{}_{C_3}\big (x^{(1)} \big)$}
\psfrag{12}{$\pi^{}_{C_3}\big (x^{(2)} \big)$}
\psfrag{13}{$\pi^{}_{C_3}\big (x^{(3)} \big)$}
\psfrag{14}{$\pi^{}_{C_3}\big (x^{(1)} \big)$}
\psfrag{15}{$\pi^{}_{C_3}\big (x^{(2)} \big)$}
\includegraphics[width = 0.7 \textwidth]{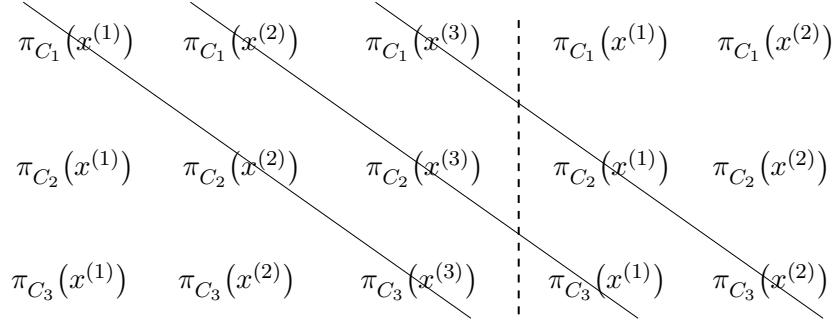}
\caption{\label{reactionscheme}
An illustration of the reaction scheme for $|\cC| = 3$. The types $x^{(1)}$, $x^{(2)}$ and $x^{(3)}$ are each split according to $\cC$ and then joined back together along the connecting lines. For the sake of clarity, the first two columns of the diagram are repeated after the vertical line. 
}
\end{figure}
First, just as in the two-parent case, the $|\cC|$ types are broken up into their subsequences $\pi_{C_j}^{} (x^{(i)})$ over the blocks of $\cC$. Then, they are arranged on a two-dimensional, $|\cC|$-periodic grid (or discrete torus), where $\pi^{}_{C_j} \big ( x^{(i)} \big )$ is placed in the $i$-th column and $j$-th row. Finally, the products are formed by joining the fragments along each diagonal line, running from north-west to south-east through the grid. Alternatively, one may think about moving the $i$-th row $i-1$ places to the left, and then joining the fragments in each column. More formally, every choice of $\cC$ and $\big (x^{(1)},\ldots,x^{(|\cC|)} \big)$ defines a reaction, 
\begin{equation} \label{gametereactions}
\sum_{j = 1}^{|\cC|} x^{(j)} \xrightarrow{\, \frac{\varrho(\cC)}{|\cC|} \,} \sum_{j=1}^{|\cC|} \bigsqcup_{i=1}^{|\cC|} \pi_{C_i}^{} \big (x^{(i+j-1)} \big),
\end{equation}
where the indices are to be read modulo $|\cC|$.

\begin{comment}
We will now show that this network, just as in the two-parent case, is strongly reversible. First, note that in the two parent case, reactions naturally come in pairs; if $|\cC| = 2$ and we replace in Eq.~\eqref{gametereactions} $\big ( x^{(1)}, x^{(2)} \big )$ by $\Big (\pi^{}_{C_1} \big (x^{(1)} \big ) \sqcup \pi^{}_{C_2} \big (x^{(2)} \big ), \pi_{C_1}^{} \big (x^{(2)}) \sqcup \pi^{}_{C_2} \big (x^{(1)} \big) \Big )$, then we obtain the reverse reaction.
For $\cC$ of arbitrary size, however, reactions no longer come in pairs, but in $|\cC|-$cycles of the form
\begin{equation*}
\ldots \xrightarrow{\, \frac{\varrho(\cC)}{|\cC|} \,} \underline{\cA}_1 \xrightarrow{\, \frac{\varrho(\cC)}{|\cC|} \,} \underline{\cA}_2 \xrightarrow{\, \frac{\varrho(\cC)}{|\cC|} \,} \ldots \xrightarrow{\, \frac{\varrho(\cC)}{|\cC|} \,} \underline{\cA}_{|\cC|} \xrightarrow{\, \frac{\varrho(\cC)}{|\cC|} \,}  \underline{\cA}_1 \xrightarrow{\, \frac{\varrho(\cC)}{|\cC|} \,} \ldots .
\end{equation*}
\end{comment}

Notice that the right-hand side depends on the order of the substrates, while the left-hand side is independent of it. For instance, in our earlier example with three sites and parents (that is, $\cC$ is $\pmin$, the trivial partition into singletons), the choice $x^{(1)} = x, x^{(2)} = y, x^{(3)} = z$ leads to Eq.~\eqref{onepossibility}, while exchanging the roles of the second and third type leads to Eq.~\eqref{anotherpossibility}.

\begin{theorem}\label{networkrepresentationgam}
For finite X, Eq.~\eqref{recoeq} is the law of mass action for the \textnormal{CRN} comprised of all reactions~\eqref{gametereactions}, one for every choice of $\cC$ and $\big (x^{(1)},\ldots,x^{(|\cC|)} \big )$. More concisely, \eqref{recoeq} is equivalent to
\begin{equation*}
\dot{\omega_t} = \sum_{\cC \in \bP(S)} \sum_{x^{(1)},\ldots,x^{(|\cC|)} \in X} \frac{\varrho(\cC)}{|\cC|} \, \omega_t \big (x^{(1)} \big) \cdot \ldots \cdot \omega_t \big (x^{(|\cC|)} \big) \Bigg ( \sum_{j = 1}^{|\cC|} \bigg ( \bigsqcup_{i = 1}^{|\cC|} \pi_{C_i}^{} \big (x^{(i + j -1)} \big) - x^{(j)} \bigg ) \Bigg ).
\end{equation*}
\end{theorem}

\begin{proof}
We will show that for all $\cC \in \bP(S)$ and all $\nu \in \cP(X)$, we have
\begin{equation*}
(\cR_\cC(\nu) - \nu) = \frac{1}{|\cC|} \sum_{x^{(1)},\ldots,x^{(|\cC|)}} \nu \big (x^{(1)} \big)\cdot \ldots \cdot \nu \big (x^{(|\cC|)} \big ) \sum_{j=1}^{|\cC|} \bigg ( \bigsqcup_{i = 1}^{|\cC|} \pi_{C_i}^{}\big (x^{(i+j - 1)}_{} \big) - x^{(j)} \bigg ).
\end{equation*}
Recall that, by Lemma~\ref{vectorrecombinator}, we have
\begin{equation*}
\begin{split}
\cR_{\cC} (\nu) &=  \sum_{x^{(1)},\ldots,x^{(|\cC|)}} \nu \big(x^{(1)} \big)\cdot \ldots \cdot \nu \big (x^{(|\cC|)} \big) \bigsqcup_{i = 1}^{|\cC|} \pi_{C_i}^{} \big (x^{(i)} \big) \\[2mm]
&= \frac{1}{|\cC|} \sum_{j=1}^{|\cC|} \sum_{x^{(1)},\ldots,x^{(|\cC|)}} \nu \big (x^{(1)} \big)\cdot \ldots \cdot \nu \big (x^{(|\cC|)} \big) \bigsqcup_{i = 1}^{|\cC|} \pi_{C_i}^{} \big (x^{(i+j-1)} \big) \\[2mm]
&= \frac{1}{|\cC|} \sum_{x^{(1)},\ldots,x^{(|\cC|)}} \nu \big (x^{(1)} \big )\cdot \ldots \cdot \nu \big (x^{(|\cC|)}\big) \sum_{j=1}^{|\cC|} \bigsqcup_{i = 1}^{|\cC|} \pi_{C_i}^{}\big (x^{(i+j-1)} \big).
\end{split}
\end{equation*}
Here, we obtain the second equality by 
replacing the product
$
\nu \big (x^{(1)} \big )\cdot \ldots \cdot \nu(x^{(|\cC|)})
$
with its cyclic permutation
$
\nu \big (x^{(1 - j +1)} \big ) \cdot \ldots \cdot \nu \big (x^{(|\cC| - j +1)} \big),
$
and subsequently renaming the indices; recall that indices are to be read modulo $|\cC|$.
Similarly, keeping in mind that $\sum_{x \in X} \nu(x) = 1$ because $\nu$ is a probability measure, we obtain,
\begin{equation*}
\begin{split}
\nu = \sum_{x \in X} \nu(x) x  =  & \frac{1}{|\cC|} \sum_{j=1}^{|\cC|} \sum_{x \in X} \bigg ( \sum_{y \in X} \nu(y) \bigg )^{j-1}\nu(x) \bigg ( \sum_{y \in X} \nu(y) \bigg )^{|\cC| - j} x \\[2mm]
= & \frac{1}{|\cC|}\sum_{x^{(1)},\ldots,x^{(|\cC|)}}  \nu \big(x^{(1)} \big ) \cdot \ldots \cdot \nu \big (x^{(|\cC|)} \big) \big (x^{(1)} + \ldots + x^{(|\cC|)} \big),
\end{split}
\end{equation*}
which completes the argument. 
\end{proof}

We have thus seen that, in the case of finite type spaces, genetic recombination can be reinterpreted as a CRN, also in the case of an arbitrary number of parents. In fact, this network is strongly reversible.

\begin{theorem}\label{reversibilitytheorem}
The \textnormal{CRN} from Theorem~\textnormal{\ref{networkrepresentationgam}} is strongly reversible in the sense of Definition~\textnormal{\ref{reversibility}}.
\end{theorem} 
\begin{proof}
Let $\cC$ be fixed. Define $\varphi : X^{|\cC|} \to X^{|\cC|}$ via
\begin{equation*}
\varphi \big (x^{(1)},\ldots,x^{(|\cC|)} \big ) \defeq \bigg (\bigsqcup_{i=1}^{|\cC|} \pi_{C_i}^{} \big (x^{(i+|\cC| - 1)} \big),\ldots,\bigsqcup_{i=1}^{|\cC|} \pi_{C_i}^{} \big (x^{(i)} \big ) \bigg ).
\end{equation*}
Note that $\varphi \big (x^{(1)},\ldots,x^{(|\cC|)} \big )$ contains the products in the reaction defined by $\cC$ together with $\big (x^{(1)},\ldots,x^{(|\cC|)} \big )$, in reverse order (compare~\eqref{gametereactions}).
As short reflection on Fig.~\ref{reactionscheme} reveals that $\varphi$ is an involution and therefore partitions $X^{|\cC|}$ into orbits that contain either one or two elements. Consider first an orbit with two elements $\big (x^{(1)},\ldots,x^{(|\cC|)} \big )$ and $\big (y^{(1)},\ldots,y^{(|\cC|)}\big )$. Then, the associated reactions form a forward-backward reaction pair,
\begin{equation*}
\sum_{j = 1}^{|\cC|} x^{(j)} \xrightarrow{\, \frac{\varrho(\cC)}{|\cC|} \,} \sum_{j = 1}^{|\cC|} y^{(j)} \text{\quad and \quad} \sum_{j = 1}^{|\cC|} y^{(j)} \xrightarrow{\, \frac{\varrho(\cC)}{|\cC|} \,} \sum_{j = 1}^{|\cC|} x^{(j)}.
\end{equation*} 
On the other hand, the reaction defined by a fixed point of $\varphi$ is void, since its product and substrate complex agree.
\end{proof}

Next, we consider the connection to gradient systems.
\section{Gradient systems}
\label{sec:gradients}
For this section, we need a few basic notions from differential (particularly Riemannian) geometry, which we recall here for the convenience of the reader. For further background, we refer the reader to~\cite{walschap}, in particular Chapter~5.
For a real-valued differentiable function $V$, defined on (some subset of) $\RR^d$, and a function $C$ with the same domain and values in the symmetric positive semi-definite matrices, we call the ordinary differential equation 
\begin{equation}\label{defgradientsystems}
\dot{x} = C(x) \nabla V(x)
\end{equation} 
a \emph{generalised gradient system} (with respect to the potential $V$).
Here, 
\begin{equation*}
\nabla \defeq \sum_{i = 1}^d \hat{e}_i \frac{\partial}{\partial x_i}
\end{equation*}
is the nabla symbol and $\{\hat{e}_1,\ldots,\hat{e}_d\}$ denotes the standard basis of $\RR^d$.

Given $x \in \RR^d$, a vector $v$ in $T_x (\RR^d)$, the tangent space of $\RR^d$ at $x$, and a continuously differentiable curve $\gamma$ in $\RR^d$ with $\gamma(0) = x$ and $\gamma'(0) = v$, recall that the \emph{directional derivative} of $V$ in direction $v$ is given by 
\begin{equation*}
\dd V(x)(v) \defeq \frac{\dd}{\dd t} V \big (\gamma(t) \big)|_{t = 0}.
\end{equation*}
The one-form $\dd V$ is called the \emph{exterior derivative} of $V$; note that it can be defined analogously for any real-valued function on a smooth manifold, and, in particular, does not depend on the Euclidean structure of $\RR^d$.
One has, by an application of the chain rule,
\begin{equation}\label{gradientdef}
\dd V(x)(v) = \sum_{j = 1}^d \gamma'(0)_j \frac{\partial}{\partial x_j} V(x) = \langle \gamma'(0), \nabla V(x) \rangle,
\end{equation}
where $\langle \cdot, \cdot \rangle$ denotes the standard scalar product on $\RR^d$. 
Replacing the standard scalar product by a general Riemannian metric $\langle \! \langle \cdot, \cdot \rangle \! \rangle_x$, (that is, a positive definite, symmetric bilinear form on the tangent space, which varies smoothly, depending on the base point),  Eq.~\eqref{gradientdef} can be used to define the \emph{gradient} of $V$ with respect to this metric~\cite[Ex.~108]{walschap}, denoted by $\text{grad}_{\langle \! \langle \cdot, \cdot \rangle \! \rangle} (V)$; it is the unique vectorfield that satisfies
\begin{equation*}
\dd V(x)(v) = \langle \! \langle v, \text{grad}_{\langle \! \langle \cdot, \cdot \rangle \! \rangle} (V)(x) \rangle \! \rangle_x
\end{equation*} 
for all $x$ and $v$.
Geometrically, this means that, unless $x$ is an equilibrium, $\text{grad}_{\langle \! \langle \cdot, \cdot \rangle \! \rangle} (V)(x)$ points in the direction of steepest ascent of $V$ at point $x$, with respect to the chosen metric. 
In particular, if $C(x)$ in Eq.~\eqref{defgradientsystems} is invertible and we consider the metric, 
\begin{equation*}
\langle \! \langle u, w \rangle \! \rangle_x := \langle u, C(x)^{-1} w \rangle,
\end{equation*}
we see that
\begin{equation*}
\text{grad}_{\langle \! \langle \cdot, \cdot \rangle \! \rangle}(V)(x) = C(x) \text{grad}_{\langle \cdot, \cdot \rangle}(V)(x) = C(x) \nabla V(x).
\end{equation*}
Thus, Eq.~\eqref{defgradientsystems} can be thought of as a gradient system in the classical sense, if we replace the Euclidean metric on $\RR^d$ by a Riemannian one, at least in the case that $C(x)$ is invertible.

The interpretation is somewhat more delicate when $C(x)$ fails to be invertible. Intuitively, one might think of the kernel of $C(x)$ as a set of forbidden directions, and try to restrict attention to submanifolds which partition the space and are in each point $x$ tangent to the image of $C$. However, this interpretation is only valid when the image of $C$ is \emph{integrable} in the sense that whenever $Y$ and $Z$ are two vectorfields such that $Y(x) \in \text{Im }C(x)$ and $Z(x) \in \text{Im }C(x)$ for all $x$, then also $[Y,Z](x) \in \text{Im }C(x)$ for all $x$, where $[Y,Z]$ denotes the Lie bracket of $Y$ and $Z$; this is the content of Frobenius' theorem \cite[Thm.~1.9.2]{walschap}.
The situation when $\text{Im } C$ is not integrable can be understood via the theory of {sub-}Riemannian manifolds. Roughly speaking, this theory is concerned with Riemannian metrics which may take the value $+\infty$; see \cite{bellaiche} for an overview.

\begin{remark}
To demonstrate that the condition of integrability is not trivial, consider the following two vectorfields on $\RR^3$.
\begin{equation*}
X_1 \defeq \frac{\partial}{\partial x_1} \text{\quad and \quad} X_2 \defeq x_1 \frac{\partial}{\partial x_3} + \frac{\partial}{\partial x_2}.
\end{equation*}
Note that 
\begin{equation*}
[X_1,X_2] = \frac{\partial}{\partial x_3}
\end{equation*}
is nowhere in the span of $X_1$ and $X_2$; thus, proving integrability in our case (and for the gradient systems arising in chemical reaction network theory in general) might be an interesting question in its own right. \hfill $\diamondsuit$
\end{remark}

We remark that, under the assumption that~\eqref{defgradientsystems} has a unique equilibrium, the potential $V$ is always a strong, (global) Lyapunov function (by which we mean that $V$ is strictly increasing along non-constant solutions). This is because
\begin{equation*}
\langle \nabla V(x), \dot{x} \rangle = \langle \nabla V(x), C(x) \nabla V(x) \rangle \geq 0,
\end{equation*}
by the positive semi-definiteness of $C(x)$. Equality holds if and only if $\nabla V(x)$ is in the kernel of $C(x)$ (implying that $\dot{x} = 0$), hence, if and only if the system is in equilibrium. % Later, we will see that a partial converse holds; every strong, global Lyapunov function can be seen as a potential of a generalised gradient system.  

We have seen in the previous section that the general recombination equation, interpreted as a chemical reaction network, is strongly reversible. Thus, it is a gradient system in the sense of Eq.~\eqref{defgradientsystems}, by standard theory; compare~\cite{yong,mielke}, where this is proved in much greater generality. For the sake of completeness, we include the simple proof of this fact, in the special case needed for our purposes.

\begin{theorem}\label{reversiblenetworksaregradientsystems}
The law of mass action for any strongly reversible \textnormal{CRN} can be written as a generalised gradient system,
\begin{equation*}
\dot{c}_t = C(c_t) \nabla F(c_t),
\end{equation*}
where
\begin{equation*}
F(c) \defeq - \sum_{s \in \cS} \Big ( c(s) \log \big (c(s) \big) - c(s) \Big)
\end{equation*}
is called the \emph{negative free energy} and $C$ is a continuous function on $\cP(\cS)$, which is smooth on its interior and takes values in the positive semi-definite matrices.
\end{theorem}

\begin{proof}
Due to strong reversibility (see Definition~\ref{reversibility}), the law of mass action takes the form
\begin{equation*}
\dot{c}_t = \sum \Big (\prod_{i = 1}^m c_t(r_i) - \prod_{i = 1}^m c_t(s_i) \Big ) \sum_{i=1}^m (s_i - r_i),
\end{equation*}
where the outer sum is taken over all forward-backward reaction pairs in the network. Define for $x,y \geqslant 0$,
\begin{equation}\label{Ldefinition}
L(x,y) := \frac{x - y}{\text{log }(x) - \text{log }(y)}.
\end{equation}
It is a straightforward exercise to verify that $L$ defines a continuous, non-negative function on $\RR_{\geqslant 0}^2$, which is smooth on $\RR_{> 0}^2$. Note that
\begin{equation*}
\nabla F(c) = - \sum_{s \in \cS} \log \big (c(s) \big) s.
\end{equation*}
Thus, setting (for each forward-backward reaction pair) 
\begin{equation*}
M(c) \defeq L \Big (\prod_{i = 1}^m c(r_i), \prod_{i = 1}^m c(s_i) \Big ) \bigg (\sum_{i=1}^m (s_i - r_i) \bigg)\bigg (\sum_{i=1}^m (s_i - r_i) \bigg)^{\! T},
\end{equation*}
we see by the multiplication rule for the logarithm, that
\begin{equation*}
\Big (\prod_{i = 1}^m c(r_i) - \prod_{i = 1}^m c(s_i) \Big ) \sum_{i=1}^m (s_i - r_i) = M(c) \nabla F(c).
\end{equation*}
Here, we also used that $s^\trans \nabla F(c)  =  - \log \big ( c(s) \big )$ for all $s \in \cS$. Since a non-negative linear combination of positive semi-definite, symmetric matrices is symmetric and positive semi-definite, the claim follows.
\end{proof}

\begin{remark}
Since the total mass, $\sum_{s \in \cS} c_t(s)$, is preserved in our case, we may replace the negative free energy $F$ in Theorem~\ref{reversiblenetworksaregradientsystems} by the entropy
\begin{equation*}
H(c) \defeq \sum_{s \in \cS} c(s) \log \big (c(s) \big).
\end{equation*} 
For the solution of the recombination equation (Eq.~\eqref{recoeq}) this has the following consequence. It is a well-known fact that, when considering the set of probability measures on a product space which all have the same marginals, the product measure of these marginals is a maximiser for the entropy. As the one-dimensional marginals are preserved under recombination (in absence of mutation or selection), the fact that Eq.~\eqref{recoeq} can be written as a generalised gradient system with respect to $H$ reflects on the fact that the solution approaches linkage equilibrium; compare~\cite[Theorem 3.1]{buerger09}.
\hfill $\diamondsuit$.
\end{remark}

\subsection{Explicit examples}
\label{subsec:explicitexample}
Combining Theorems~\ref{reversiblenetworksaregradientsystems},\ref{networkrepresentationgam} and~\ref{reversibilitytheorem}, for finite $X$, there exists a Function $C$, defined on $\cP(X)$ with values in the symmetric positive semi-definite matrices such that
\begin{equation*}
\dot{\omega}_t^{} = C(\omega_t^{}) \nabla F(\omega_t^{})
\end{equation*}
is equivalent to the recombination equation~\eqref{recoeq}. Our goal is now to write down the function $\nu \mapsto C(\nu)$ for $\nu \in \cP(X)$ explicitly for concrete examples.
The most simple one is the classical case with two parents and two diallelic loci (compare~\cite[Ex.~1]{muellerhofbauer}). Then, we have the reaction
\begin{equation*}
(0,0) + (1,1) \xleftrightarrow{\, \varrho \,} (1,0) + (0,1).
\end{equation*}
Identifying $(0,0)$ with the first, $(0,1)$ with the second, $(1,0)$ with the third and $(1,1)$ with the fourth basis vector in $\RR^4$, the matrix $C(\nu)$, as constructed in the proof of Theorem~\ref{reversiblenetworksaregradientsystems} can be written as
\begin{equation*}
\varrho L \big ( \nu(0,0) \nu(1,1), \nu(1,0) \nu(0,1) \big )
\begin{pmatrix}
1 & -1 & -1 & 1 \\
-1 & 1 & 1 & -1 \\
-1 & 1 & 1 & -1 \\
1 & -1 & -1 & 1
\end{pmatrix},
\end{equation*}
with $L$ defined in Eq.~\eqref{Ldefinition}

Next, we treat the slightly more complicated example of three diallelic loci (but still $2$ parents); compare~\cite[Ex.~2]{muellerhofbauer}. Again, we denote the two alleles by $0$ and $1$. We denote the type $(i_1,i_2,i_3)$ by $g^{}_{4i_i + 2 i_2 + i_3}$; in other words, the index of a type is just the type itself, read as a binary integer. For example, we refer to $(0,0,0)$ by $g_0$ and to $(1,0,1)$ by $g_5$, and identify $g_i$ with the canonical $i+1$-th basis vector of $\RR^8$.

 Now, by the proof of Theorem~\ref{reversiblenetworksaregradientsystems}, we associate to each reaction pair of the form
\begin{equation} \label{somepairofreactions}
g_{i_1}^{} + g_{i_2}^{} \xleftrightarrow{\, \kappa \,} g_{j_1}^{} + g_{j_2}^{},
\end{equation}
an $8 \times 8$ matrix $M(\nu)$ with entries
\begin{equation*}
M_{ij}(\nu) \defeq \begin{cases}
\kappa L \big (\nu(g_{i_1}^{}) \nu(g_{i_2}^{}), \nu(g_{j_1}^{}) \nu(g_{j_2}^{}) \big ), &\text{ if } g_{i-1} \text{ and } g_{j-1} \text{ are on the same side of \eqref{somepairofreactions}, } \\
-\kappa L \big (\nu(g_{i_1}^{}) \nu(g_{i_2}^{}), \nu(g_{j_1}^{}) \nu(g_{j_2}^{}) \big ), &\text{ if } g_{i-1} \text{ and } g_{j-1} \text{ are on different sides of \eqref{somepairofreactions}, } \\
0, & \text{otherwise}
\end{cases}
\end{equation*}
and $C(\nu)$ is then given by summing these matrices over all forward-backward reaction pairs in the network.
To keep things tidy, instead of summing over all forward-backward reaction pairs, we write down the sums over each different linkage class seperately; this allows to take advantage of the following symmetry implied by our choice of indices. Namely, as $1$s are only exchanged between gametes but their relative positions in the sequence remains unchanged, the sum of indices is the same for each complex that are in the same linkage class, of which there are seven; six consisting of only one forward-backward reaction pair  each, and one consisting of six such pairs.
Assume that $M$ belongs to a reaction within a complex where the indices sum to $\ell$. Then, it is easy to see that we have $M_{i,j} = M_{\ell - i + 2,j} = M_{i, \ell-j + 2} = M_{\ell - i + 2,\ell -j + 2}$. This means that, for $\ell$ odd, $M$ is of the form 
\begin{equation*}
\left (
\begin{array}{c|c|c}
A &  \baro \!  A & 0 \\ \hline
\minuso \! A & \minuso \! \baro \! A & 0 \\ \hline 
0 & 0 & 0
\end{array}
\right ) \text{ if } \ell \leqslant 7 \text{ and }
\left (
\begin{array}{c|c|c}
0 & 0 & 0 \\ \hline
0 & \minuso \! \baro \! A &  \minuso \!  A \\ \hline
0 &\baro \! A & A 
\end{array}
\right ) \text{ for } \ell > 7,
\end{equation*} 
where $\baro$ denotes the reversal of columns and $\minuso$ denotes the reversal of rows within a matrix and $A$ is a
$\frac{\ell + 1}{2} \times \frac{\ell + 1}{2}$ matrix if $\ell \leqslant 7$ and a $\frac{14 - \ell +1}{2}$ matrix if $\ell > 7$. For $\ell$ even , M is of the form
\begin{equation*}
\left (
\begin{array}{c|c|c|c}
A & 0 & \baro \!  A & 0 \\ \hline
0 & 0 & 0 &0 \\ \hline
\minuso \! A & 0 & \minuso \! \baro \! A & 0 \\ \hline 
0 & 0 & 0 & 0
\end{array}
\right )\text{ if } \ell \leqslant 7 \text{ and } 
\left (
\begin{array}{c|c|c|c}
0 & 0 & 0 & 0 \\ \hline
0 & \minuso \! \baro \! A & 0 & \minuso \! A \\ \hline
0 & 0 & 0 & 0 \\ \hline
0 & \baro \! A & 0 &  A
\end{array}
\right )  \text{ for } \ell > 7,
\end{equation*}
where $A$ is now an $\frac{\ell}{2} \times \frac{\ell}{2}$ matrix if $\ell \leqslant 7$ and $\frac{14 - \ell}{2}$ if $\ell > 7$; Here, the extra $0$ between the reflected copies of $A$ comes from the fact that reactions of the form 
\begin{equation*}
g_{i}^{} + g_{i}^{} \xleftrightarrow{\, \kappa \,} g_{i}^{} + g_{i}^{}
\end{equation*}
do not contribute to the system. Let us now write these matrices $A$ for the different linkage classes. We  abbreviate the function $\nu \mapsto L \big (\nu(g_{i_1}^{})  \nu(g_{i_2}^{}), \nu(g_{j_1}^{}) ) \nu(g_{j_2}^{}) \big )$ by $L_{i_1i_2,j_1j_2}$. For all $1 \leqslant i \leqslant 3$, $\varrho_i^{}$ denotes the recombination rate for the partition $\{\{i\}, \{1,2,3\} \setminus \{i\} \}$. For the first six linkage classes in~\cite[Ex.~2]{muellerhofbauer}, each consisting of one reaction, we have in place of $A$
\begin{equation*}
\begin{split}
&\left (\begin{smallmatrix} %for 0 + 6 <-> 4 + 2
(\varrho_1^{} + \varrho_2^{} )L_{06,24} & 0 & -(\varrho_1^{} + \varrho_2^{} )L_{06,24} \\
0 & 0 & 0 \\
-(\varrho_1^{} + \varrho_2^{} )L{06,24} & 0 & (\varrho_1^{} + \varrho_2^{} )L_{06,24}
\end{smallmatrix} \right ),
\left (\begin{smallmatrix} %for 1 + 7 <-> 5 + 3 
(\varrho_1^{} + \varrho_2^{} )L_{17,35} & 0 & -(\varrho_1^{} + \varrho_2^{} )L_{17,35} \\
0 & 0 & 0 \\
-(\varrho_1^{} + \varrho_2^{} )L_{17,35} & 0 & (\varrho_1^{} + \varrho_2^{} )L_{17,35}
\end{smallmatrix} \right ), \\ 
&\left (\begin{smallmatrix} %for 0 + 5 <-> 4 + 1
(\varrho_1^{} + \varrho_3^{} )L_{05,14} & -(\varrho_1^{} + \varrho_3^{} )L_{05,14} & 0 \\ 
-(\varrho_1^{} + \varrho_3^{} )L_{05,14} & (\varrho_1^{} + \varrho_3^{} )L_{05,14} & 0 \\
0 & 0 & 0
\end{smallmatrix} \right ),
\left ( \begin{smallmatrix} % for 2 + 7 <-> 6 + 3
0 & 0 & 0 \\
0 & (\varrho_1^{} + \varrho_3^{} )L_{27,36}  & - (\varrho_1^{} + \varrho_3^{} )L_{27,36} \\
0 & - (\varrho_1^{} + \varrho_3^{} )L_{27,36} & (\varrho_1^{} + \varrho_3^{} )L_{27,36}
\end{smallmatrix} \right ),
\end{split}
\end{equation*}
representing the reactions $g_{0}^{} + g_{6}^{} \xleftrightarrow{\, \varrho_1^{} + \varrho_2^{} \,} g_{4}^{} + g_{2}^{}, \,
g_{1}^{} + g_{7}^{} \xleftrightarrow{\, \varrho_1^{} + \varrho_2^{} \,} g_{5}^{} + g_{3}^{}, \,
g_{0}^{} + g_{5}^{} \xleftrightarrow{\, \varrho_1^{} + \varrho_3^{} \,} g_{4}^{} + g_{1}^{}, \,
g_{2}^{} + g_{7}^{} \xleftrightarrow{\, \varrho_1^{} + \varrho_3^{} \,} g_{6}^{} + g_{3}^{}$
and
\begin{equation*} 
\left ( \begin{smallmatrix} % for 0 + 3 <-> 2 + 1
(\varrho_2^{} + \varrho_3^{} )L_{03,12} & -(\varrho_2^{} + \varrho_3^{} )L_{03,12} \\
-(\varrho_2^{} + \varrho_3^{} )L_{03,12} & -(\varrho_2^{} + \varrho_3^{} )L_{03,12}
\end{smallmatrix} \right),
\left ( \begin{smallmatrix} % for 4 + 7 <-> 5 + 6
(\varrho_2^{} + \varrho_3^{} )L_{47,56} & -(\varrho_2^{} + \varrho_3^{} )L_{47,56} \\
- (\varrho_2^{} + \varrho_3^{} )L_{47,56} & (\varrho_2^{} + \varrho_3^{} )L_{47,56}
\end{smallmatrix} \right ),
\end{equation*}
representing the reactions $g_{0}^{} + g_{3}^{} \xleftrightarrow{\, \varrho_2^{} + \varrho_3^{} \,} g_{2}^{} + g_{1}^{}$ and 
$g_{4}^{} + g_{7}^{} \xleftrightarrow{\, \varrho_2^{} + \varrho_3^{} \,} g_{5}^{} + g_{6}^{}$. Finally, the last linkage class, comprised of the six reactions $g_2^{} + g_5^{} \xleftrightarrow{\, \varrho_1^{} \,} g_{6}^{} + g_{1}^{}$, $g_6^{} + g_1^{} \xleftrightarrow{\, \varrho_2^{} \,} g_4^{} + g_3^{}$, $g_{4}^{} + g_{3}^{} \xleftrightarrow{\, \varrho_1^{} \,} g_0^{} + g_7^{}$, $g_0^{} + g_7^{} \xleftrightarrow{\, \varrho_2^{} \,} g_2^{} + g_5^{}$, $g_2^{} + g_5^{} \xleftrightarrow{\, \varrho_3^{} \,} g_4^{} + g_3^{}$, $g_6^{} + g_1^{} \xleftrightarrow{\, \varrho_3^{} \,} g_0^{} + g_7^{}$, is represented by 
\begin{equation*}
\left ( \begin{smallmatrix}
\varrho_1^{} L_{07,34}  + \varrho_2^{} L_{07,25} + \varrho_3^{} L_{16,07} & - \varrho_3^{} L_{16,07} & -\varrho_2^{} L_{07,25} & -\varrho_1^{} L_{07,34} \\
-\varrho_3^{} L_{16,07} &  \varrho_1^{} L_{16,25} + \varrho_2^{} L_{16,34} + \varrho_3^{} L_{16,07} & -\varrho_1^{} L_{16,25} & -\varrho_2^{} L_{16,34} \\
-\varrho_2^{} L_{25,07} & -\varrho_1^{} L_{25,16} & \varrho_1^{} L_{25,16} + \varrho_2^{} L_{25,07} + \varrho_3^{} L_{25,34} & - \varrho_3^{} L_{25,34} \\
-\varrho_1^{} L_{34,07} & -\varrho_2^{} L_{34,16} & -\varrho_3^{} L_{34,25} & \varrho_1^{} L_{34,07} + \varrho_2^{} L_{34,16} +  \varrho_3^{} L_{34,25} 
\end{smallmatrix} \right ).
\end{equation*}

\begin{comment}
\begin{equation*}
\left (
\begin{smallmatrix}
C_1  & 
 \varrho_3^{} L(07,61)  & \varrho_2^{} L(07,25) & \varrho_1^{} L(34,07)  &  \varrho_1^{} L(34,07) & \varrho^{}_3 L(25,07) &   \varrho^{}_3 L(61,07) & C_1 \\
 \varrho_3^{} L(07,61) & C_2 & \varrho_1^{} L(25,61) & \varrho_2^{} L(61,43) & \varrho_2^{} L(61,43) &\varrho_1^{} L(25, 61) &C_2 & \varrho_3^{} L(07, 16) \\
 \varrho_2^{} L(07,25) & \varrho_1^{} L(25,61) &C_3 &\varrho_3 L(25,43) &\varrho_3 L(25,43) & C_3 &\varrho_1^{} L(25,61) &\varrho_2^{} L(07,25) \\
\varrho_1^{} (L(34,07) & \varrho_2^{} L(61,43) & \varrho_3^{} L(25,43) &C_4 &C_4 &\varrho_3^{} L(25,43) & \varrho_2^{} (16,43) & \varrho_1^{} L (43,07) \\
\varrho_1^{} L(34,07) & \varrho_2^{} L(61,43) & \varrho_3^{} L(25,34) & C_5 & C_5 &\varrho_3^{} L(25,34) &\varrho_2^{} L(61,34) & \varrho_1^{} L(34, 07) \\
\varrho_2^{} L(07,25) & \varrho_1^{} L(25,61) & C_6 & \varrho_3^{} L(25,34) & \varrho_3^{} L(25,34) & C_6 & \varrho_1^{} L(25,61) & \varrho_2^{} L(07,25) \\
\varrho_3^{} L(07,61)& C_7 & \varrho_1^{} L(25,61) & \varrho_2^{} L(61,43) & \varrho_2^{} L(61,43) & \varrho_1^{} L(25,61) & C_6 & \varrho_3^{} L(07,61) \\
C_7 & \varrho_3^{} L(07,61) & \varrho_2 L(07,25) &\varrho_1^{} L(34,07) & \varrho_1^{} L(34,07) &\varrho_2^{} L(07,25) & \varrho_3^{} L(07,61) & C_7
\end{smallmatrix}
\right )
\end{equation*}
\end{comment}

\section{monotone Markov chains and the partitioning process}
\label{sec:markov} 
We have seen how the result of M\"uller and Hofbauer~\cite{muellerhofbauer} generalises in the setting of an arbitrary number of parents, at least for finite type spaces. For more general, potentially uncountable type spaces, this approach fails because it is not clear how to even make sense of the notion of the concentration of individual types, unless $\omega_t$ is pure point. Now, we show how the evolution of the law of the partitioning process, related to $\omega$ via Eq.~\ref{stochasticrepresentation} can be written as a gradient system. Ultimately, this is due to the monotonicity of its sample paths; recall from~\eqref{markovgenerator} that the transition rate from $\cA$ to $\cB$ vanishes whenever $\cB \not \preccurlyeq \cA$. In particular, the number of blocks increases strictly in each transition. 
\begin{deff}\label{MCsmo}
Let $\cX = \big (\cX_t \big)_{t \geqslant 0}$ be a continuous-time Markov chain on a finite state space $E$ with rate matrix $\big( Q(i,j) \big )_{i,j \in E}$; it is called a \emph{Markov chain with strictly monotone orbits} (MCsmo)(with respect to a real-valued function $W$ on $E$) if $Q(i,j) > 0$ implies that $W(j) > W(i)$. \hfill $\diamondsuit$
\end{deff}

Recall that the distribution of a finite-state Markov chain $\cX = (\cX_t)_{t \geqslant 0}$ can be interpreted as a probability vector,
\begin{equation*}
p^\cX_t \defeq \sum_{i \in E} p^\cX_t(i) i, 
\end{equation*} 
which evolves in time according to the differential equation,
\begin{equation}\label{distributionode}
\dot{p}_t^\cX = \sum_{i \in E} \sum_{j \in E} p_t^\cX(i) Q(i,j) (j - i).
\end{equation}

If $\cX$ has strictly monotone orbits in the sense of Definition~\ref{MCsmo}, Eq.~\eqref{distributionode} can be written as a generalised gradient system, as defined in Section~\ref{sec:gradients}.

\begin{theorem}
Let $\cX = \big (\cX_t \big)_{t \geqslant 0}$ be a \textnormal{MCsmo} with respect to $W$ and define $\Psi : \RR^E \to \RR$,
\begin{equation*}
\Psi(p) \defeq \sum_{i \in E} p(i) W(i).
\end{equation*}
Then, Eq.~\eqref{distributionode}, which describes the time evolution of $p_t^\cX$, is equivalent to
\begin{equation*}
\dot{p}_t^\cX = K(p_t^\cX) \nabla \Psi(p_t^\cX),
\end{equation*}
where $K$ takes values in the symmetric, positive semi-definite matrices, is continuous on $\cP(E)$ and smooth on its interior.
\end{theorem}

\begin{proof}
Define
\begin{equation*}
K(p) \defeq \sum_{\substack{i,j \in E \\ Q(i,j) > 0}} \frac{p(i) Q(i,j)}{W(j) - W(i)} (j-i)(j-i)^\trans.
\end{equation*}
Since $\Psi$ is linear with (constant) gradient
\begin{equation*}
\nabla \Psi = \sum_{i \in E} W(i) i,
\end{equation*}
we have $(j - i)^\trans \nabla \Psi = W(j) - W(i)$ and thus,
\begin{equation*}
K(p) \nabla \Psi =  \sum_{\substack{i,j \in E \\ Q(i,j) > 0}} \frac{p(i) Q(i,j)}{W(j) - W(i)} \big (W(j) - W(i) \big) (j-i) = \sum_{i,j \in E} p(i) Q(i,j) (j - i).
\end{equation*}
Inserting $p_t^\cX$ for $p$, this is exactly the right-hand side of Eq.~\eqref{distributionode}
\end{proof}

\begin{comment}
\begin{remark}\label{jordan}
It was shown in~\cite{baakebaakesalamat} that  Eq.~\eqref{nonlinearodes} has, in certain degenerate cases, resonant solutions of the form $t \mapsto t^k e^{\lambda t}$. These are the same solutions as for a linear differential equation with non-diagonalisable coefficient matrix (and therefore non-vanishing curl). This appears to be a contradiction to our result that Eq.~\eqref{nonlinearodes} can be written as a generalised gradient system since it is well known that the curl of every classical gradient system vanishes! However, the question of existence of a \emph{generalised} gradient structure is more subtle. Consider, for example, the vectorfield $(x,y) \mapsto (-y,x)$ which has constant, nonvanishing curl. In polar coordinates, it becomes constant and therefore  obviously a gradient field. In other words, we have a generalised gradient field on $\RR^2 \setminus \RR_{\leq 0}$ where the metric is given by the coordinate change. Of course, this gradient structure can not be defined on the entire plane since the flow of this field along any closed curve around the origin will be nonzero, independent of the metric chosen. It seems that a similar phenomenon occurs for the recombination field. \hfill $\diamondsuit$
\end{remark}
\end{comment}

The partitioning process mentioned in Section~\ref{sec:recoeq} is a process of succesive refinement; in every non-silent transition, the number of blocks increases at least by one. Thus, it is a MCsmo with respect to the number of blocks.
\begin{coro}\label{PPasMCsmo}
The law $p^\Sigma_{}$ of the partitioning process with generator $\cQ$ given in Eq.~\eqref{markovgenerator} satisfies a generalised gradient system with respect to $N$ given by
\begin{equation*}
N(p) = \sum_{\cA \in \bP(S)} p(\cA) |\cA|. 
\end{equation*}
\end{coro}

We conclude with an explicit example. 
\begin{example}\label{jordan}
Let us consider a Markov chain with $4$ states $A,B,C,D$ and jump rates $q(A,B) = q(A,C) = 1$ and $q(B,D) = q(C,D) = 2$. All other transition rates are $0$. This is a  Markov chain with strictly monotone orbits in the sense of Definition \ref{MCsmo}, with respect to $W$ given by $W(A) = 1, W(B) = W(C) = 2, W(D) = 3$.  Upon identifying $A,B,C,D$ with the standard basis of $\RR^4$, the linear differential equation describing the dynamics of its distribution reads
\begin{equation}\label{original}
\setlength{\arraycolsep}{0.8mm}
\def\arraystretch{0.7}
\dot{p}_t = \begin{pmatrix}
-2 & 0 & 0 & 0 \\
1 & -2 & 0 & 0 \\
1 & 0 & -2 & 0 \\
0 & 2 & 2 & 0
\end{pmatrix} p_t,
\end{equation}
and can be rewritten as
\begin{equation}\label{rewritten}
\setlength{\arraycolsep}{0.8mm}
\def\arraystretch{0.7}
\dot{p}_t = \begin{pmatrix}
2p_t(A) & - p_t(A) & -p_t(A) & 0 \\
-p_t(A) & p_t(A) + 2p_t(B) & 0 & -2p_t(B) \\
-p_t(A) & 0 & 2p_t(C) + p_t(A) & -2p_t(C) \\
0 & -2p_t(B) & -2p_t(C) & 2p_t(C) + 2 p_t(B)
\end{pmatrix}  \begin{pmatrix}
1 \\ 2 \\ 2 \\ 3
\end{pmatrix}.
\end{equation}
Here, the vector $(1,2,2,3)^\trans$ is the gradient (with respect to the euclidean metric) of 
\begin{equation*}
\Psi(p) = p(A) + 2p(B) + 2 p(C) + 3 p(D).
\end{equation*}
Also, the matrix is symmetric and it is positive semi-definite, as it can be written as a sum of positive semi-definite matrices (as long as $p(A),p(B),p(C) \geq 0)$,
\begin{equation*}
p(A) \left (\begin{smallmatrix} -1 \\ 1 \\ 0 \\ 0 \end{smallmatrix} \right ) \left ( \begin{smallmatrix} -1 & 1 & 0 & 0 \end{smallmatrix} \right ) + p(A) \left (\begin{smallmatrix} -1\\0\\1\\0\end{smallmatrix} \right ) \left ( \begin{smallmatrix} -1 & 0 & 1 & 0 \end{smallmatrix} \right ) + 2p(B) \left ( \begin{smallmatrix} 0\\-1\\0\\1 \end{smallmatrix} \right ) \left ( \begin{smallmatrix} 0 & -1 & 0 & 1 \end{smallmatrix} \right ) + 2p(C) \left( \begin{smallmatrix} 0\\0\\-1\\1 \end{smallmatrix} \right) \left ( \begin{smallmatrix} 0 & 0 & -1 & 1 \end{smallmatrix} \right ),
\end{equation*}
evaluated at $p = p_t$. Thus, Eq.~\eqref{rewritten} is a generalised gradient system in the sense of Eq.~\eqref{defgradientsystems}. 
\end{example}
Note that the coefficient matrix in Eq.~\eqref{original} is not diagonalisable; this is because the eigenvalue $-2$ has algebraic multiplicity $3$, but the associated eigenspace is merely two-dimensional and spanned by $(0,1,-1,0)^\trans$ and $(0,1,0,-1)^\trans$. Its general solution will therefore contain terms of the form $t \ee^{-2 t}$. This seems to contradict the fact that a linear generalised gradient system can not have resonant solutions of the form $t^k \ee^{\lambda t}$ for $k \geq 1$. One has to keep in mind, however, that the gradient representation only holds on the nonnegative cone (which is forward-invariant for the system). Note also that the problematic generalised eigenspace only has a trivial intersection with $\RR_{\geqslant 0}^4$. 
We conclude with one additional example.
\begin{example}
Let us now consider the actual partitioning process, for three loci. We have the five partitions $\cA_1 =\{\{1,2,3\}\}, \cA_2 = \{\{1\},\{2,3\}\}, \cA_3 = \{\{1,3\},\{2\}\}, \cA_4 = \{\{1,2\},\{3\}\}$ and $\cA_5 = \{\{1\},\{2\},\{3\}\}$. Identifying $\cA_i$ with the $i$-th basis vector in $\RR^5$, the generator $\cQ$ of the partitioning process (cf. Eq.~\eqref{markovgenerator}) reads
\begin{equation*}
\left (\begin{smallmatrix}
-\varrho_1^{} - \varrho_2^{} - \varrho_3^{} & \varrho_1^{} & \varrho_2^{} & \varrho_3^{} & 0 \\
0 & - \varrho_2^{} - \varrho_3^{} & 0 & 0 & \varrho_2^{} + \varrho_3^{} \\
0 & 0 & - \varrho_1^{} - \varrho_3^{} & 0 & \varrho_1^{} + \varrho_3^{} \\
0 & 0 & 0 & -\varrho_1^{} - \varrho_2^{} & \varrho_1^{} + \varrho_2^{} \\
0 & 0 & 0 & 0 & 0
\end{smallmatrix}\right ),
\end{equation*}
where $\varrho_1^{}, \varrho_2^{}, \varrho_3^{}$ are as in Subsection~\ref{subsec:explicitexample},
and he gradient system then for the distribution $p_t^{\Sigma}$ then reads
\begin{equation*}
\dot{p}_t^\Sigma = 
\left ( \begin{smallmatrix}
D_1 & -p_t^{}(\cA_1) \varrho_1^{} & -p_t^{} (\cA_1) \varrho_2^{} &  -p_t^{} (\cA_1) \varrho_3^{} & 0 \\
 -p_t^{}(\cA_1) \varrho_1^{} & D_2 & 0 & 0 & -p_t^{} (\cA_2) (\varrho_2^{} + \varrho_3^{}) \\
-p_t^{}(\cA_1) \varrho_2^{} & 0 & D_3 & 0 & - p_t^{} (\cA_3) (\varrho_1^{} + \varrho_3^{}) \\
-p_t^{} (\cA_1) \varrho_3^{} & 0 & 0 & D_4 & -p_t^{}(\cA_4)(\varrho_1^{} + \varrho_2^{}) \\
0 & -p_t^{}(\cA_2)(\varrho_2^{} + \varrho_3^{}) & - p_t^{}(\cA_3) (\varrho_1^{} + \varrho_3^{}) & - p_t^{}(\cA_4) (\varrho_1^{} + \varrho_2^{}) & D_5
\end{smallmatrix} \right ) 
\left ( \begin{smallmatrix}
1 \\ 2 \\ 2 \\ 2 \\ 3
\end{smallmatrix} \right ),
\end{equation*}
where $D_1,\ldots,D_5$ are chosen such that the rows sum to $0$ and $(1,2,2,3)^\trans$ is the gradient $\nabla N$ of the mean number of blocks $N$, defined in Corollary~\ref{PPasMCsmo}. Again, the maximum of the potential, the partition $\{\{1\},\{2\},\{3\}\}$ characterises linkage equilibrium (`all sites come from independent ancestors').
\end{example}

\section{Nonlinear partitioning as a chemical reaction network}
\label{sec:nonlinearpartitioning}
We have seen in the previous chapter that the evolution of the law of the partitioning process can be rewritten as a linear generalised gradient system.  We now consider the nonlinear system from Theorem~\ref{reduction}. We will see that it, too, can be interpreted as the law of mass action for a network of chemical reactions between the partitions of $S$. Its construction is very similar to the network from Section~\ref{sec:networks}. 

To motivate this result, imagine that at time $t = 0$, we paint every gamete in a different color. As described in Theorem~\ref{gametereactions} and Fig.~\ref{reactionscheme}, for every $\cC \in \bP(S)$, every randomly chosen $|\cC|$-tuple of gametes undergoes a chemical reaction as in Eq.~\eqref{gametereactions}
at rate $\frac{\varrho(\cC)}{|\cC|}$. But now, instead of investigating the effect on the type distribution, we ask how the initially assigned colors are mixed in the process. To this end, we attach to each individual a partition of its sites by grouping together all sites with the same color.

Now, consider the $j$-th gamete that results from such a reaction (compare Eq.~\eqref{gametereactions}); for two sites $k$ and $\ell$ in this individual to have the same color, they must come from the same individual on the left-hand side (this is due to the fact that the tuple was chosen randomly and, as there are infinitely many colors in the population, the probability that the same color occurs in more than one individual in the chosen sample is negligible). More formally, there must be an $i$ between $1$ and $|\cC|$ such that $k$ and $\ell$ are both in $C_i$. If that is true, both sites come from the $i+j-1$-th individual, and thus must share the same block of $\cA_{i+j-1}$. Put more concisely, this means that $k$ and $\ell$ belong to the same block of the induced partition $\cA_{i+j-1}^{}|_{C_i}$ for some $i \in \{1,\ldots,|\cC|\}$. Equivalently, this means that the partition that describes the coloring of the $j$-th product gamete is given precisely by
\begin{equation*}
\bigcup_{i = 1}^{|\cC|} \cA_{i + j - 1}^{}|_{C_i}.
\end{equation*}
For an illustration, see Fig.~\ref{reactionwithsplit}.
Thus, the reaction network from Section~\ref{sec:networks} translates to the system consisting of the reactions
\begin{equation} \label{coeffreactions}
\sum_{j=1}^{|\cC|} \cA_j \xrightarrow{\, \frac{\varrho(\cC)}{|\cC|} \,} \sum_{j=1}^{|\cC|} \bigcup_{i=1}^{|\cC|} \cA_{i+j-1}|_{C_i},
\end{equation}
one for each $\cC$ and every $|\cC|$-tuple of partitions of $S$; as always, indices are to be read mod $|\cC|$. These reactions are of the same form as the ones between gametes in Eq.~\eqref{gametereactions}, after replacing the type fragments $\pi_{C_i}^{} \big (x^{(i+j-1)} \big )$ with the induced partitions $\cA_{i + j - 1}^{}|^{}_{C_i}$.

We finish by showing that the law of mass action of this chemical reaction network is precisely the nonlinear system from Theorem~\ref{reduction}.
\begin{theorem}\label{networkrepresentationcoeff}
The nonlinear system of ordinary differential equations that describes the dynamics of the coefficients in~\eqref{ansatz} can be written as the law of mass action for the \textnormal{CRN} comprised of all reactions~\eqref{coeffreactions}. More concisely, \eqref{nonlinearodes} is equivalent to
\begin{equation*}
\dot{a}_t = \sum_{\cC} \sum_{\cA_1,\ldots,\cA_{|\cC|}} \frac{\varrho(\cC)}{|\cC|} a_t (\cA_1) \cdot \ldots \cdot a_t(\cA_{|\cC|})  \Bigg ( \sum_{j = 1}^{|\cC|} \bigg ( \bigcup_{i=1}^{|\cC|} \cA_{i + j - 1}|_{C_i}- \cA_j \bigg ) \Bigg ),
\end{equation*}
where the summation is over $\bP(S)$.
\end{theorem}

\begin{proof}
We will use the following identity (the proof of which will conclude the proof of the theorem),
\begin{equation}\label{uglyproduct}
\prod_{i=1}^{|\cB|}\sum_{\substack{\cC \in \bP(S) \\ \cC|_{B_i} = \cA|_{B_i}}} a(\cC)  = \frac{1}{|\cB|} \sum_{j=1}^{|\cB|} \sum_{\cA_1,\ldots,\cA_{|\cB|}} \delta \bigg (\cA, \bigcup_{i = 1}^{|\cB|} \cA_{i+j-1}|_{B_i} \bigg) \cdot a(\cA_1) \cdot \ldots \cdot a(\cA_{|\cB|}),
\end{equation}
valid for all $\cB \succcurlyeq \cA$ and all $a \in \RR^{\bP(S)}$,
where $\cB = \{B_1,\ldots,B_{|\cB|}\}$.
Inserting~\eqref{uglyproduct}, we see that the second sum on the right-hand side of Eq.~\eqref{nonlinearodes},
\begin{equation*}
\sum_{\cA}\sum_{\udo{\cB} \succcurlyeq \cA} \Bigg ( \prod_{i=1}^{|\cB|}\sum_{\substack{\cC \in \bP(S) \\ \cC|_{B_i} = \cA|_{B_i}}} a_t(\cC) \Bigg) \varrho (\cB) \cA,
\end{equation*}
can be written as 
\begin{equation}
\sum_{\cA} \sum_{\udo{\cB} \succcurlyeq \cA} \Bigg ( \frac{\varrho(\cB)}{|\cB|}\sum_{j =1}^{|\cB|} \sum_{\cA_1,\ldots,\cA_{|\cB|}} \delta \bigg (\cA,\bigcup_{i = 1}^{|\cB|} \cA_{i+j-1}|_{B_i}\bigg ) a_t(\cA_1)\cdot \ldots \cdot a_t(\cA_{|\cB|}) \cA \Bigg ).
\end{equation} 
Notice that the second argument of the Kronecker function is always finer than $\cB$. Thus, the whole summand vanishes whenever $\cB \succcurlyeq \cA$ does not hold. We may therefore ignore the restriction $\cB \succcurlyeq \cA$ in the inner sum, which allows us then to change the order of summation. After using the Kronecker function to perform the summation with respect to $\cA$, what remains is
\begin{equation*}
\sum_{\cB} \frac{\varrho(\cB)}{|\cB|} \sum_{\cA_1,\ldots,\cA_{|\cB|}} a(\cA_1)\cdot \ldots \cdot a(\cA_{|\cB|}) \sum_{j =1}^{|\cB|}\bigcup_{i = 1}^{|\cB|} \cA_{i+j-1}|_{B_i}.
\end{equation*}
Up to renaming $\cB$ with $\cC$, this is exactly the first part of the law of mass action for the CRN described above.  Using the same argument as in the proof of Theorem~\ref{networkrepresentationgam}, the first sum in Eq.~\eqref{nonlinearodes},
\begin{equation*}
-\sum_{\cB} \varrho(\cB)\sum_{\cA} a(\cA)\cA,
\end{equation*}
can be rewritten as 
\begin{equation*}
-\sum_{\cB} \frac{\varrho(\cB)}{|\cB|} \sum_{\cA_1,\ldots,\cA_{|\cB|}} a_t(\cA_1) \cdot \ldots \cdot a_t(\cA_{|\cB|}) (\cA_1 + \ldots + \cA_{|\cB|}).
\end{equation*}
Up to renaming $\cB$ with $\cC$, this completes the proof, provided Eq.~\eqref{uglyproduct} is correct. To show this, we start by expanding the right hand side,
\begin{equation*}
\begin{split}
\prod_{i=1}^{|\cB|}\sum_{\substack{\cC \in \bP(S) \\ \cC|_{B_i} = \cA|_{B_i}}} a(\cC)& = \sum_{(\cA_1,\ldots,\cA_{|\cB|}) \in \cG(\cA)} a(\cA_1) \cdot \ldots \cdot a(\cA_{|\cB|}) \\
&= \sum_{\cA_1,\ldots,\cA_{|\cB|}} \delta \bigg (\cA, \bigcup_{i = 1}^{|\cB|} \cA_{i}|_{B_i} \bigg)a(\cA_1) \cdot \ldots \cdot a(\cA_{|\cB|})
\end{split}
\end{equation*}
where $\cG(\cA)$ is the set of all $|\cB|$-tupels $(\cA_1,\ldots,\cA_{|\cB|})$ of partitions with $\cA_i|_{B_i} = \cA|_{B_{i}}$.
Since $\cA \preccurlyeq \cB$ implies that 
\begin{equation*}
\cA = \cA|_{B_1} \cup \ldots \cup \cA|_{B_{|\cB|}},
\end{equation*} 
$(\cA_1,\ldots,\cA_{|\cB|}) \in \cG(\cA)$ if and only if 
\begin{equation*}
\cA = \bigcup_{i=1}^{|\cB|} \cA_{i}|_{B_i}.
\end{equation*}
Now, as in the proof of Theorem~\ref{networkrepresentationgam}, we replace the product (for $1 \leq j \leq |\cB|$)
\begin{equation*}
a(\cA_1) \cdot \ldots \cdot a(\cA_{|\cB|})
\end{equation*}
by
\begin{equation*}
a(\cA_{1 - j + 1}) \cdot \ldots \cdot a(\cA_{|\cB| - j + 1})
\end{equation*}
and subsequently rename the summation indices. Thus,
\begin{equation*}
\begin{split}
\sum_{\cA_1,\ldots,\cA_{|\cB|}} \delta \bigg (\cA, \bigcup_{i = 1}^{|\cB|} \cA_{i}|_{B_i} \bigg)a(\cA_1) \cdot \ldots \cdot a(\cA_{|\cB|}) \\ = \frac{1}{|\cB|}\sum_{j=1}^{|\cB|} \sum_{\cA_1,\ldots,\cA_{|\cB|}} \delta \bigg (\cA, \bigcup_{i = 1}^{|\cB|} \cA_{i + j -1}|_{B_i} \bigg)a(\cA_1) \cdot \ldots \cdot a(\cA_{|\cB|}),
\end{split}
\end{equation*}
which finishes the proof of Eq.~\eqref{uglyproduct} and hence, of the theorem.
\end{proof}

\begin{figure}[h]
\psfrag{rate}{$\frac{\varrho(\cC)}{|\cC|}$}
\psfrag{plus}{\hspace{0.5mm}$+$}
\psfrag{plus2}{\raisebox{-1mm}{\hspace{0.5mm}$+$}}
\includegraphics[width=0.5 \textwidth]{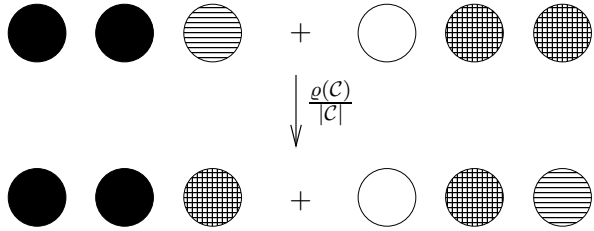}
\caption{\label{reactionwithsplit}
A reaction between two gametes with $3$ loci, corresponding to the partition $\cC = \{\{1,2\},\{3\}\}$. This means that the leading two sites of the left gamete on the top is combined with the trailing third site of the gamete on the right, and the leading two sites of the gamete to the right are combined with the trailing third of the left one. Here, the coloring of the sites is represented by different patterns. The partitions associated associated with the gametes are as follows. For the  substrate complex (top), we have $\cA_1 = \{\{1,2\},\{3\}\}$ and $\cA_2 = \{\{1\},\{2,3\}\}$, and the product complex (bottom) consists of 
$\cA_1|_{C_1} \cup \cA_2|_{C_2} = \{\{1,2\}\} \cup \{\{3\}\} = \{\{1,2\},\{3\}\}$ and
$\cA_1|_{C_2} \cup \cA_2|_{C_1} = \{\{3\}\} \cup \{\{1\},\{2\}\} = \{\{1\},\{2\},\{3\}\}$.
}
\end{figure}

Despite their similar appearance, there is one crucial difference between the CRN from Section~\ref{sec:networks}, and the one above. Because the products are pieced together from partitions of subsets induced by the substrates, the total number of blocks on the right-hand side is in general strictly larger than on the left-hand side. This implies that this network is \emph{not} reversible, and the question whether it can be interpreted as a gradient system remains open. 
The loss of reversibility appears to be the coarse-graining of the information in our system that we performed by transitioning from the (potentially infinite) set of types to the finite set of partitions. This is vaguely reminiscent of the common phenomenon in statistical mechanics where the projection of the underlying (high-dimensional) microscopic model to a smaller set of macroscopic degrees of freedom leads to a loss of reversibility.

\section*{Acknowledgements}

It is a pleasure to thank M. Baake, J. Hofbauer and C. Wiuf for helpful discussions. The thoughtful comments of two anonymous referees helped to improve the presentation and are thankfully acknowledged. This
work was supported by the German Research Foundation (DFG), 
within the SPP 1590.

\end{document}